\documentclass[11pt]{amsart}

\usepackage{amsmath,amssymb,amsthm,amscd, amsxtra, amsfonts}

\usepackage[all]{xy}

\usepackage{palatino}

\newtheorem{theorem}{Theorem}[section]
\newtheorem{lemma}[theorem]{Lemma}

\newtheorem{proposition}[theorem]{Proposition}
\newtheorem{corollary}[theorem]{Corollary}

\newtheorem{remark}{Remark}

\def \<{\langle}
\def \>{\rangle}

\def \a{\alpha }

\newcommand{\bea}{\begin{eqnarray}}
\newcommand{\eea}{\end{eqnarray}}
\newcommand{\be}{\begin {equation}}
\newcommand{\ee}{\end{equation}}

\newcommand{\wt}{{\rm {wt} }   }

\newcommand{\Z}{\Bbb Z}

\newcommand{\W}{\mathcal W}
\newcommand{\mip}{\overline{M(1)}}

\newcommand{\smip}{\overline{SM(1)}}

\newcommand{\NS}{\frak{ns} }

\newcommand{\N}{{\Bbb Z}_{\ge 0} }
\newcommand{\C}{\Bbb C}

\newcommand{\WW}{\boldsymbol{ \mathcal{W}}}

\newcommand{\la}{\langle}
\newcommand{\ra}{\rangle}

\newcommand{\triplet}{\mathcal{W}(p)}
\newcommand{\striplet}{\mathcal{SW}(m)}

\newcommand{\hf}{\mbox{$\frac{1}{2}$}}
\newcommand{\thf}{\mbox{$\frac{3}{2}$}}
\newcommand{\vak}{\bf 1}

\usepackage{amssymb}

\begin{document}
  \title[The structure of Zhu's algebras for certain $\W$-algebras] {The structure of  Zhu's algebras for certain $\W$-algebras}
\author{Dra\v{z}en Adamovi\'c and Antun Milas}
\address{Department of Mathematics, University of Zagreb, Croatia}
\email{adamovic@math.hr}

\address{Department of Mathematics and Statistics,
University at Albany (SUNY), Albany, NY 12222}
\email{amilas@math.albany.edu}

\thanks{The second author was supported in part by NSA and NSF grants.}

\begin{abstract}
We introduce a new approach that allows us to determine the structure of Zhu's algebra for certain vertex operator (super)algebras which admit horizontal $\mathbb{Z}$-grading. By using this method and an earlier description of Zhu's algebra for the singlet $\W$-algebra, we completely describe the structure of Zhu's algebra for the triplet vertex algebra $\triplet$. As a consequence, we prove that Zhu's algebra $A(\triplet)$ and the related Poisson algebra $\mathcal{P}(\triplet)$ have the same dimension.  We also completely describe Zhu's algebras for the $N=1$ triplet vertex operator superalgebra
$\striplet$. Moreover, we obtain similar results for the $c=0$ triplet vertex algebra $\WW_{2,3}$, important in logarithmic conformal field theory.
Because our approach is "internal"  we had to employ several constant term identities for purposes of getting right upper bounds on ${\rm dim}(A(V))$.

This work is, in a way, a continuation of the results published in \cite{AdM-triplet}.

\end{abstract}

\maketitle

\section{Introduction}

In this work we address two important algebraic objects that can be associated
to any conformal vertex (super)algebra $V$, both introduced in a seminal paper by Zhu \cite{Z}:

\begin{itemize}

\item[(i)] Zhu's associative algebra $A(V)$, and

\item[(ii)] commutative Poisson algebra $\mathcal{P}(V)=V/C_2(V)$.

\end{itemize}

The Zhu's  algebra $A(V)$ is instrumental in representation theory of vertex algebra and has been a subject of numerous
papers. On the other hand, $\mathcal{P}(V)$ is primarily used for purposes of modular invariance of graded
dimensions \cite{Z}. The two algebras are of course  closely related; we always have a natural surjective map from $\mathcal{P}(V)$ to ${\rm gr}A(V)$ (the associated
graded algebra of $A(V)$), giving
\be \label{ineq}
{\rm dim}(\mathcal{P}(V)) \geq {\rm dim}(A(V)),
\ee
at least if $\mathcal{P}(V)$ is finite-dimensional (i.e. $V$ is $C_2$-cofinite).

Fairly recently, Gaberdiel and Gannon \cite{GG} have initiated a thorough study of possible relationships between
$\mathcal{P}(V)$ and $A(V)$, by focusing  primarily to rational  vertex algebras of affine type.
Although they have observed that for many familiar examples - such as Virasoro minimal models - these two algebras will have
the same dimension (and thus the above map will be an isomorphism), there are many instances for which this is false
(take for instance $V$ to be the level one affine vertex algebra associated to Lie algebra of type $E_8$).
The main observation is that the discrepancy between two algebras is somewhat  "controlled" by  twisted $V$-modules.
They also provided a conjecture for the equality of dimensions for certain vertex algebras associated to representations
of affine algebras at positive integral levels. Some of their conjectures have been recently settled  in  \cite{FFL} and \cite{FL}.

In view of (\ref{ineq}) we can also contemplate whether $\mathcal{P}(V)$ and ${\rm gr} A(V)$ are isomorphic for $V$ being $C_2$-cofinite, but not
necessarily rational. Such vertex algebras have recently attracted a lot of attention in connection
with logarithmic conformal field theory \cite{HLZ},\cite{Miy}. But even for the well-known triplet vertex algebra $\triplet$ (cf. \cite{FGST-triplet}, \cite{FHST}, \cite{AdM-triplet}),
the two algebras have not been determined in full yet for all $p$ (see however \cite{AdM-triplet}) .
The same issue can be also addressed for $C_2$-cofinite vertex operators superalgebras.
Again, very little is known apart from vertex algebras associated to $N=1$ minimal models (cf. \cite{A-1997} and \cite{M}).

The aims of this paper include:  (a) We present a new approach for determining Zhu's algebra for vertex algebra which admit horizontal $\mathbb{Z}$-grading
(and some additional properties). This method is successfully applied to triplet vertex algebras. (b) We settle several  conjectures
from our previous work needed for better description of $\mathcal{P}(V)$ in the case of the triplet algebras $\triplet$.
(c) Finally, we show how to extend our results to vertex operator superalgebras.

So let us briefly outline the main results.

We closely follow \cite{AdM-triplet} and \cite{AdM-striplet}, and as before let $\triplet$ be the triplet vertex operator algebra
and $\striplet$ be the $N=1$ supertriplet algebra. From the structural results about Zhu's algebras from  \cite{AdM-triplet}, \cite{AdM-tstriplet} and \cite{AdM-striplet}, we have that for the complete  description of Zhu's algebras $A(\triplet)$, $A(\striplet)$ and $A_{\sigma}(\striplet)$ one has to describe the center of these algebras. Let us describe our approach in the case of triplet vertex algebra. Let $\mip \subset \triplet$ be the singlet vertex operator algebra (cf. \cite{A-2003}, \cite{AdM-2007}). We investigate a natural homomorphism of Zhu's algebras
$\Phi : A(\mip) \rightarrow A(\triplet)$ and identify  the center of $A(\triplet)$ as a subalgebra of   $A(\mip) /  \ \mbox{Ker}(\Phi)  $. We prove that the kernel $\mbox{Ker}(\Phi)$ is a principal ideal in Zhu's algebra $A(\mip)$, and  show that the dimension of
$A(\mip) /  \ \mbox{Ker}(\Phi)  $  is $ 4 p -1$. This  easily gives the description of the center and proves  that the dimension of  $A(\triplet)$ is $6 p -1$. We also prove that:
\begin{theorem} For $p \geq 2$,
$$\dim A( \triplet) = \dim \mathcal{P}(\triplet).$$
\end{theorem}

For vertex operator superalgebras (with suitable grading) the situation is a bit different because there are four distinct algebras of interest here: the usual untwisted Zhu's algebra $A(V)$, its $\sigma$-twisted
counterpart $A_{\sigma}(V)$, the Poisson superalgebra $\mathcal{P}(V)$ and its even subalgebra $\mathcal{P}_0(V)$. It can be shown that there is again a surjective map
from $\mathcal{P}_0(V)$ to ${\rm gr}A(V)$, giving an upper bound on ${\rm dim}(A(V))$. But more important object of consideration turns out to be $A_{\sigma}(V)$.
By applying similar approach we show that the dimension of  $A(\striplet)$ is $ 6 m +1$ and the dimension of $A_{\sigma} (\striplet)$ is $12 m + 8$. In this way we present a positive answer on  conjectures from papers \cite{AdM-tstriplet}-\cite{AdM-striplet}.

\begin{theorem} For all $m \in \mathbb{N}$,
$$\dim A_{\sigma}( \striplet) = \dim \mathcal{P}(\striplet).$$
\end{theorem}
This equality is known to hold for $N=1$ vertex superalgebras associated to $N=1$ minimal models (cf. \cite{M}).

Interestingly, the natural homomorphism  from  $\mathcal{P}_0(\striplet)$ onto ${\rm gr}(A(\striplet))$ yields a nontrivial kernel.
Namely, we have
\begin{corollary} For all $m \in \mathbb{N}$,
\be
\dim A(\striplet) < \dim \mathcal{P}_0(\striplet).
\ee
\end{corollary}
Our methods can be of course applied for "logarithmic extension" of $(p,q)$ Virasoro minimal models (cf. \cite{AdM-2009}, \cite{FGST-log}, \cite{GRW}). Thus in Section \ref{log-ext-min} we completely describe the Zhu algebra for the vertex algebra $\WW_{2,3}$ of central charge zero. As a consequence, we prove the following result, predicted in the physics literature
\begin{corollary} The vertex algebra $\WW_{2,3}$ admits
a logarithmic module of $L(0)$ nilpotent rank $3$ (here $L(0)$ denotes a generator of the Virasoro algebra).
\end{corollary}

It is important to observe that our approach is "internal" in a sense that only properties of the vertex algebra in question have been used to determine Zhu's algebra. We do not use any information
about the structure or existence of logarithmic representations (which is in general nontrivial), nor modular invariance \cite{AdM-triplet}, \cite{Miy}. Instead, the existence of logarithmic modules
is obtained from complete structure of the Zhu algebra. In this way we give additional evidence for the correspondence between the category of modules  for triplet vertex algebras and the category of modules for certain quantum groups which are Kazhdan-Lusztig duals of triplet vertex algebras (cf. \cite{FGST-triplet}, \cite{FGST-log-kl}).

\vskip 5mm

\noindent {\bf Acknowledgment:} We thank D. Svrtan and D. Zeilberger for useful discussion. We thank the referee for her/his valuable comments and suggestions.

\section{Main definitions}

The starting point for this paper is to recall the definition of Zhu's algebra for vertex operator
(super)algebras following \cite{KW}, \cite{Z}.

Let $(V=  V ^{\bar 0} \oplus V ^{\bar 1},Y, {\bf 1}, \omega)$ be a vertex operator superalgebra. We
shall always assume that
\bea
&&V^{\bar 0}=\coprod_{ n \in {\N} } V(n), \quad V ^{\bar 1} = \coprod_{ n \in \tfrac{1}{2} +{\N} } V(n) \nonumber \\
&&\mbox{where} \ \  V(n) = \{ a \in V \ \vert \ L(0) a = n v \}. \nonumber \eea
For $a \in V(n)$, we shall write $\wt (a) = n$, or ${\rm deg}(a)=n$.
As usual, vertex operator  associated to $a \in V$ is denoted by $Y(a,x)$, with the mode expansion
$$Y(a,x)=\sum_{n \in \mathbb{Z}} a_n x^{-n-1}.$$

We define two bilinear maps $* : V  \times V \rightarrow V$,
$\circ : V \times V \rightarrow V$ as follows: for homogeneous $a,
b \in V$, let
\bea
a* b &&= \left\{\begin{array}{cc}
 \  \mbox{Res}_x Y(a,x) \frac{(1+x) ^{\deg (a)}}{x}b  & \mbox{if} \ a,b  \in V^{\bar{0}} \\
  0 & \mbox{if} \ a \ \mbox{or} \ b  \in V^{\bar{1}} \
\end{array}
\right.  \\
a\circ b &&= \left\{\begin{array}{cc}
 \  \mbox{Res}_x Y(a,x) \frac{(1+x) ^{\deg (a)} }{x^2}b  & \mbox{if} \ a  \in V^{\bar{0}} \\
  \  \mbox{Res}_x Y(a,x) \frac{(1+x) ^{\deg (a) -\tfrac{1}{2} }}{x}b  & \mbox{if} \ a   \in V^ {\bar{1}} \
\end{array}
\right. \eea

Next, we extend $*$ and $\circ$ to $V \otimes V$ linearly, and
denote by $O(V)\subset V$ the linear span of elements of the form
$a \circ b$, and by $A(V)$ the quotient space $V / O(V)$. The
space $A(V)$ has a unitary associative algebra structure, with the multiplication induced by $*$. Algebra $A(V)$ is called  the
Zhu's algebra of $V$. The image of $v \in V$, under the natural
map $V \mapsto A(V)$ will be denoted by $[v]$.

For a homogeneous $a \in V$ we define
$$o(a) = a_{\wt(a)-1}.$$
In the case when $V ^{\bar 0} = V$, $V$ is a vertex operator algebra and we get the usual definition of Zhu's algebra for vertex operator algebras.

With $V$ as above, we let $$C_2(V)=\langle a_{-2} b : a,b \in V \rangle \ \ {\rm and} \  \ \mathcal{P}(V)=V/C_2(V).$$
The quotient space $\mathcal{P}(V)$ has an algebraic structure of a commutative Poisson
algebra \cite{Z}. Explicitly, if we denote by $\bar{a}$ the image of $a$ under the natural map $V \mapsto \mathcal{P}(V)$ the Poisson bracket is given by
$\{\overline{a}, \overline{b} \} = \overline{ a_0 b}$ and commutative product $\overline{a} \cdot  \overline{b} = \overline{a_{-1} b}$. If $V$ is a vertex superalgebra, clearly $\mathcal{P}(V)$ is $\mathbb{Z}_2$-graded.
Its even part will be denoted by $\mathcal{P}_0(V)$. From the given definitions it is not hard to construct an increasing filtration of $A(V)$ such that ${\rm gr}A(V)$ maps
onto $\mathcal{P}(V)$.

\section{A homomorphism of Zhu's algebras }

\label{hom-zhu}

 Assume that $V$ is a vertex operator superalgebra which admits the (horizontal) ${\Z}$--gradation:
 $$ V = \bigoplus_{\ell \in {\Z}} V_{\ell}, \quad V_{\ell_1} \cdot V_{\ell_2} \subset V_{\ell_1 + \ell_2}, $$
 where
 $$V_{k} \cdot V_{\ell}:=\mbox{span}_{\C} \{ u_n v : u \in V_{k}, v \in V_{\ell}, n \in \mathbb{Z} \}.$$
 In addition, assume there is $G  \in {\rm End}(V)$ such that:

 \bea
 &&\label{uv-1}  G \quad \mbox{is a derivation on V}, \\
 &&\label{uv-2} G (V_{\ell})  \subset V_{\ell +1},  \\
 && \label{uv-3} \omega \in V_0, \quad  [G, Y(\omega,z)] = 0, \\
 &&  \label{uv-4} G \vert V_{\ell} \quad \mbox{is injective for } \ \ell < 0,  \\
 &&\label{uv-5} G \vert V_{\ell} \quad \mbox{is surjective  for } \ \ell \ge 0.
 \eea

 Clearly, $V_0$ is a subalgebra of $V$.

 We shall consider Zhu's algebras
 $$ A(V_0) = V_0 / O(V_0), \quad A(V) = V / O(V). $$
 Let
 $$ O(V)_0 = O(V) \cap V_0 = \mbox{span}_{\C} \{ v \circ w \ \vert \ v \in V_{\ell}, \ w \in V_{-\ell}, \ \ell \in {\Z} \} . $$

 Consider the following homomorphism of Zhu's algebras
 \bea
 \Phi : && A(V_0) \rightarrow A(V) \nonumber \\
  && v + O(V_0) \mapsto v + O(V), \quad v \in V_0 . \nonumber \eea
 Let $$ A_0 (V) = \mbox{Im} (\Phi) = \frac{V_0}{ O(V)_0}. $$
 Clearly, $$ A_0 (V) \cong A(V_0) / \mbox{Ker} (\Phi).$$

 We are interested in $\mbox{Ker}(\Phi)$. First we notice that
 $$ \mbox{Ker}(\Phi) = \{ u + O(V_0) \in A(V_0) \ \vert \ u \in O(V)_0 \}.$$

 \begin{lemma} \label{lema-0}We have
 $$ O(V)_0 \subset  G(V_{-1}) + O(V_0). $$
 \end{lemma}

 \begin{proof}

 We have to prove that for every $\ell \in {\Z}$:
\bea && v \in V_{\ell},  \quad w \in V_{-\ell} \quad \implies v \circ w \in G(V_{-1}) + O(V_0). \label{claim} \eea

We  consider   the case when $ \ell \in {\N}$. The case $ \ell < 0$ can be proved analogously.

We  prove the claim (\ref{claim}) by induction on $ \ell \in {\N}$. For $\ell =0$ the claim holds.

Assume that $ \ell =1$, $v \in V_1$, $w \in V_{-1}$.

Take  $v ' \in V_0$ such that $ v = G v'$. Since $ v ' \circ w \in V_{-1}$ and $ v ' \circ G w \in O(V_0)$, we have
$$ v \circ w = G ( v' \circ w) - v ' \circ G w \in  G(V_{-1})  +   O(V_0). $$

Assume now that the claim holds for $\ell >0$. Let
$$v \in V_{\ell+1}, \quad  w \in V_{-\ell-1}. $$

Take  $v ' \in V_{\ell}$ such that  $ G v' = v$. By using induction hypothesis and  the fact that  $ v ' \circ w \in V_{-1}$,  we get
$$ v \circ w = G ( v' \circ w) - v ' \circ G w \in G(V_{-1}) + O(V_0). $$
The proof follows.
\end{proof}

Lemma \ref{lema-0} implies the following result:

\begin{proposition} \label{prop-1}
Let  $ [G(V_{-1})] = \{ [G v], \ v \in V_{-1}  \}$. Then
$$ \mbox{Ker}(\Phi)     \subset [G(V_{-1})]. $$
\end{proposition}

\section{ Structure of Zhu's algebra $A(\triplet)$ }

In this section we shall describe the structure of Zhu's algebra of the triplet vertex algebra $\triplet$. We use several structural results on triplet and singlet vertex algebra obtained in \cite{A-2003}, \cite{AdM-2007} and \cite{AdM-triplet}, which we recall briefly.

 \vskip 2mm

Let $p \in {\Z}$ such that $p \ge 2$. As usual, $V_L$ will denote the lattice vertex algebra associated to the positive definite even lattice
$$L = {\Z} \a, \quad \la \a, \a \ra = 2 p .$$
Let $Y$ be the associated vertex operator.
For details of the construction see, for instance, \cite{AdM-triplet}.

The triplet vertex algebra $\triplet$, of {\em type}  $(2,2p-1,2p-1,2p-1)$,  is a vertex subalgebra of $V_L$ generated by
the conformal vector
$$ \omega = \frac{1}{4 p} \a(-1)^2  + \frac{p-1}{2p} \a (-2), $$
and the {\em primary} vectors
$$\ F = e ^{-\a}, \ H = Q F, \ E = Q ^2 e ^{-\a},$$
where $ Q= e ^{\a}_0= \mbox{Res}_z Y(e ^{\a},z)$ is a so-called {\em screening} operator.
The operator $Q$ acts horizontally (preserving conformal weight) so that
$${\rm deg}(X)=2p-1, \  \ X \in \{E,F,H \}.$$
There is another useful description of $\triplet$. As a module for the Virasoro algebra, $V_L$ is not completely reducible. But it has a semisimple submodule which is isomorphic to the triplet vertex algebra.
More precisely,
\bea \triplet &=& {\rm soc}_{Vir}(V_L) \nonumber \\
&=& \bigoplus_{n= 0} ^{\infty}  \bigoplus_{j= 0} ^{2n} U(Vir). Q ^j e^{-n \a}. \label{dec-triplet} \eea

The vertex subalgebra of $\triplet$ generated by $\omega$ and $H$ is called singlet vertex algebra and will be denoted by $\mip$. Clearly, $\mip$ is a subalgebra of the Heisenberg vertex algebra $M(1)$.

For $ i \in {\Z}$ we set  $$h_{i,1} = \frac{(p-i) ^2 - (p-1) ^2}{4p}.$$
The triplet $\triplet$ is known to be $C_2$-cofinite but irrational \cite{AdM-triplet}. Moreover,
$\triplet$ has precisely $2 p$ inequivalent irreducible modules which are usually denoted by
$$ \Lambda(1), \dots, \Lambda(p), \Pi(1), \dots, \Pi(p). $$
For $1 \le i \le p$, the
top component of  $\Lambda(i)$ is $1$-dimensional and has conformal weight $h_{i,1}$,  and the top component of $\Pi(i)$ is $2$--dimensional with conformal  weight
$h_{3p-i,1}$.

\begin{theorem}  \cite{A-2003} \label{zhu-alg}
Zhu's associative algebra $A(\mip)$ is isomorphic to the
commutative algebra \\  ${\C}[ x, y] / \la P(x,y) \ra$,
where  $\la P(x,y) \ra $ is the principal ideal generated by
\bea \label{ass-poly} P(x,y) = y ^{2} - \frac{( 4 p) ^{2 p -1} }{
(2 p -1)! ^{2}} \ ( x + \frac{(p-1) ^{2}}{4 p}) \prod_{i= 0}
^{p-2} \left( x + \frac{i}{4 p} ( 2 p - 2 -i) \right) ^{2} . \eea
(Here $x$ and $y$ correspond to $[\omega]$ and $[H]$).
\end{theorem}

\begin{remark} \label{after-zhu-alg}
Description in Theorem \ref{zhu-alg} shows that the subalgebra $P$
of $A(\mip)$ generated by $[\omega]$ is isomorphic to $\mathbb{C}[x]$ and
that $1$ and $[H]$ are algebraically independent over $\mathbb{C}[x]$, meaning that if
$$A([\omega])+B([\omega])*[H]=0 \ \ {\rm in} \ \ A(\mip),$$
for some polynomials $A$ and $B$
then $A(x)=B(x)=0$.
\end{remark}

 Recall that as a module over Virasoro algebra $\triplet$ is generated by singular vectors
 $$ \{ Q ^j  e ^{-n \a}, \ n \in {\N}, 0 \le j \le 2n \}, $$
 and $\mip$ is generated by singular vectors
  $$ \{ Q ^n  e ^{-n \a}, \ n \in {\N} \}. $$

For every $ \ell \in {\N}$, we define
$$ \triplet_{-\ell} = \mip . e ^{-\ell \a} = \bigoplus_{n = 0} ^{\infty} U(Vir) . Q ^n e ^{-(n+\ell) \a },  $$
$$ \triplet_{\ell} = \mip . Q ^{ 2 \ell}  e ^{-\ell \a} = \bigoplus_{n = 0} ^{\infty} U(Vir) . Q ^{n + 2 \ell} e ^{-(n+\ell) \a}  . $$
Then we have yet another description of $\triplet$.
\begin{proposition}
For every $\ell \in {\Z}$, $\triplet_{\ell}$ is an irreducible $\mip$--module and
$$ \triplet = \bigoplus_{\ell \in {\Z}} \triplet _{\ell}.$$
Moreover, for $v \in \triplet _{\ell_1}$, $w \in \triplet_{\ell_2}$, we have
$$Y(v,z ) w \in \triplet_{\ell_1 + \ell_2} ((z)).$$
Operator $G= Q$ satisfies conditions (\ref{uv-1})-(\ref{uv-5}).
\end{proposition}
\begin{proof}
Irreducibility has been established in \cite{AdM-2007}, and the decomposition follows from description given in \cite{AdM-triplet}.
\end{proof}

\begin{lemma}
 We have
 $$O(\triplet)_0 \subset \bigoplus_{n=1} ^{\infty} U(Vir). Q^n e ^{-n\a}  + O(\mip). $$
 \end{lemma}
\begin{proof}
First we notice that
$$ Q (\triplet_{-1}) =   \bigoplus_{n=1} ^{\infty} U(Vir). Q^n e ^{-n\a}. $$
Now assertion follows from Lemma \ref{lema-0}.
\end{proof}

\begin{lemma} \label{lema-2}
Assume  $n \ge 1$. Then in Zhu's algebra $A(\mip)$ we have
$$[ Q ^n e ^{-n \a}]   = A([\omega]) * [H]  + B([\omega]) * f_p ([\omega]), $$
where
$$f_p(x) = \prod_{i=1} ^{3p-1} (x- h_{i,1}),$$
and   $A, B \in {\C}[x]$.
\end{lemma}

\begin{proof}
Let $Y(H,z) = \sum_{ j \in \Z} H_j z ^{-j-1}$.
The results from \cite{A-2003} and \cite{AdM-2007} imply that
$$ H_{j} Q ^k e ^{-k \a} \in U(Vir) Q ^{k+1} e ^{-(k+1) \a} \oplus U(Vir)Q ^{k-1} e ^{ -(k-1) \a}, $$
where $j \in {\Z}$ and $ k \ge 1$. Moreover there exists $j_0 \le -2$ such that
$$ H_{j_0} Q ^k e ^{-k \a} = C Q ^{k+1} e ^{-(k+1) \a} + f(\omega) Q ^{k-1} e ^{ -(k-1) \a} \quad (C \ne 0, \ f(\omega) \in U(Vir)).$$
This easily implies that in $A(\mip)$
$$ [Q ^{k+1} e ^{-(k+1) \a}] = \overline{f}([\omega]) * [ Q ^{k-1} e ^{ -(k-1) \a}]$$
for certain $\overline{f} \in {\C}[x]$. By induction, we now have that
$$[ Q ^n e ^{-n \a}]   = A([\omega]) * [H]  + B([\omega]) *[Q ^2 e ^{-2 \a}]. $$
 From \cite{AdM-triplet} we have that $ [Q ^2 e ^{-2 \a}] = D f_p([\omega])$, $D \ne 0$. The proof follows.
\end{proof}

Now we shall use results from Section \ref{hom-zhu} and obtain the following important result.

\begin{theorem} \label{jezgra} We have
\bea \mbox{Ker} ( \Phi) &=&A(\mip) .( p([\omega]) * [H] ) \nonumber \\ &\cong&  \mbox{span} _{\C} \{ A([\omega]) *p([\omega])* [H] + B([\omega]) * f_p([\omega]), \quad A, B \in {\C}[x] \}, \nonumber \eea
where
$$p(x) = \prod_{i=2p} ^{3p-1} (x- h_{i,1}), \quad f_p(x) = \prod_{i=1} ^{3p-1} (x- h_{i,1}).$$
\end{theorem}
\begin{proof} By using Proposition \ref{prop-1} and  Lemma \ref{lema-2} we conclude
$$ \mbox{Ker} (\Phi) \subset \{ A([\omega]) * [H] + B([\omega]) * f_p([\omega]), \quad A, B \in {\C}[x] \}. $$
Let $ u \in  \mbox{Ker} (\Phi)$. Then
$$ u = A([\omega]) * [H] + B([\omega]) * f_p([\omega])$$
for some $A, B \in {\C}[x]. $

Since $\mbox{Ker}(\Phi)$ is an ideal in $A(\mip)$ we have
$$[H]*u=\tilde{A}([\omega]) * [H] + \tilde{B}([\omega]) * f_p([\omega]),$$
for some $\tilde{A}, \tilde{B} \in {\C}[x]$. On the other hand
$$ [H] * u = A([\omega]) * [H] ^2 + B([\omega]) * f_p([\omega]) * [H] \in  \mbox{Ker}(\Phi). $$
Because of
$$A([\omega]) * [H] ^2 = A([\omega]) * g ([\omega]), \quad g(x)= C_p \prod_{i=1} ^{2p-1} (x-h_{i,1}),$$
and Remark \ref{after-zhu-alg} we conclude that $\tilde{B}(x)f_p(x)=\tilde{B}(x)p(x)g(x)=A(x)g(x)$.
Therefore, $p(x) \vert A(x)$. In this way we have proved that

$$  \mbox{Ker} (\Phi) \subseteq \{ A_1([\omega]) * p([\omega]) * [H] + B_1([\omega]) * f_p([\omega]), \quad A_1, B_1 \in {\C}[x] \}. $$
Now the results from \cite{AdM-triplet} imply the relation
$$ \Phi \biggl(  p([\omega]) * [H] \biggr)  = \Phi \biggl(f_p([\omega] ) \biggr) =0,$$
which proves the opposite inclusion. The proof follows.
 \end{proof}

\vskip 5mm

Let $$A_0 (\triplet) = \mbox{Im} (\Phi) \cong A(\mip) / \mbox{Ker} ( \Phi). $$

The results from \cite{AdM-triplet} (cf. Theorem 5.1 of \cite{AdM-triplet}) imply that
$$ A(\triplet) = A_{-1} (\triplet) \oplus A_0 (\triplet) \oplus A_1 (\triplet),$$
where $$ A_{-1} (\triplet) = A_0 (\triplet) . [F], \quad   A_{1} (\triplet) = A_0 (\triplet) . [E], $$
and
$$ \dim A_{\pm 1} (\triplet) =p, \quad \dim A_0( \triplet) \le 4p-1.$$
Now Theorem \ref{jezgra} implies  that $\dim A_0 (\triplet) = 4p-1. $
In this way we have proved the following theorem:

\begin{theorem} \label{dim-zhu-triplet} For $p \geq 2$, we have
$$ \dim A(\triplet) = 6 p -1. $$
\end{theorem}

In \cite{AdM-triplet}, we have proved that $\dim A(\triplet) \le 6 p -1$ and that $\dim A(\triplet) = 6p -1$ if and only if $A(\triplet)$ contains $2$-dimensional ideals $\mathbb{I}_{h_{i,1}}$ parameterized by conformal weights $h_{i,1}$ (see Section 5 of \cite{AdM-triplet}). These ideals give $(p-1)$ indecomposable $2$-dimensional representations. Then our  Theorem \ref{dim-zhu-triplet} implies that $A(\triplet)$ has these  indecomposable representations.
By using Zhu's correspondence we obtain the following result.

\begin{corollary}
For every $1 \le i \le p-1$, there exits the logarithmic, self-dual,  $\N$--graded $\triplet$--module $\mathcal{P}_{i} ^+$ such that the top component $\mathcal{P}_{i} ^+ (0)$ is two-dimensional and
$L(0)$ acts on   $\mathcal{P}_{i} ^+ (0)$ as
$$\left(
    \begin{array}{cc}
      h_{i,1} & 1 \\
      0 & h_{i,1} \\
    \end{array}
  \right) .$$

\end{corollary}

\begin{remark}
The existence of logarithmic modules (in the case $p$ is a prime number) was proved in \cite{AdM-triplet} by using modular invariance. Here we have presented a proof which use only theory of Zhu's algebras. We show that existence of logarithmic modules can detected directly from the internal structure of Zhu's algebra.
\end{remark}

\begin{theorem} For every $p \geq 2$
$${\rm dim}(\mathcal{P}(\mathcal{W}(p)) \leq 6p-1.$$
\end{theorem}
\begin{proof}
The  description of $C_2(\triplet)$ from \cite{AdM-triplet} gives that  $\mathcal{P}(\mathcal{W}(p))$ is generated by
$$\overline{\omega}, \overline{H}, \overline{E}, \overline{F},$$
(here $\overline{a}$ denotes the image of $a$ in $\mathcal{P}(\triplet)$),
and that the following relations hold:
\bea \label{c2-1}
&&\overline{\omega} ^ {3 p -1} =  \overline{E} ^2 = \overline{F} ^2 = \overline{H} \cdot \overline{E} = \overline{H}  \cdot \overline{F} = 0, \\
&& \overline{H} ^2 = - \overline{E} \cdot \overline{F} = {\nu}  \overline{\omega} ^{ 2p -1} \quad (\nu  \ne 0).
\eea
By using (\ref{dec-triplet}) and the fact that $\mbox{wt}( H_{-2} F)  =4p-1$ we get
$$ H_{-2} F = \nu L(-2) ^p  F + v_1, \quad v_1 \in C_2 (\triplet), \quad \nu \in {\C}. $$
Applying the screening operator  $Q$ and using the fact that $C_2(\triplet)$ is $Q$--invariant, we get
$$ E_{-2} F = \nu L(-2) ^p H  +v_2, \quad E_{-2} H = \nu L(-2) ^p E  +v_3$$
where $v_2, v_3 \in C_2 (\triplet)$.

Another important ingredient is Theorem \ref{identity-triplet} below, which implies that in $A(\mip)$ we have the following relation
$$ [E \circ F ] = \nu_1  p([\omega]) * [H], \qquad (\nu_1 \ne 0).$$
But this relation in Zhu's algebra easily gives that $\nu = \nu_1 \ne 0$.

Therefore we have proved
\bea \label{c2-2} &&\overline{\omega} ^ p \overline{H}=\overline{\omega} ^ p \overline{E}= \overline{\omega} ^p\overline{F}=0. \eea
Finally,  relations (\ref{c2-1}) -- (\ref{c2-2}) prove the inequality
$${\rm dim}(\mathcal{P}(\mathcal{W}(p)) \leq 6p-1.$$
\end{proof}
Combined together
\begin{corollary} We have
$$ \dim A(\triplet) = \dim (\mathcal{P}(\triplet))=6p-1. $$
\end{corollary}

\section{On Zhu's algebra $A(\WW_{2,3})$ }
\label{log-ext-min}

In this section we shall consider the triplet vertex algebra $\WW_{2,p}$ from \cite{AdM-2009} and \cite{FGST-log}.

Assume $p$ is an odd natural number, $p \ge 3$, and  let $$L = {\Z}
\alpha, \ \ \ \la \alpha , \alpha \ra = p.$$  Let $V_L$ be the associated vertex superalgebra.
We recall that $V_L$ is generated by vectors $ e^{ \a}$ and $e ^{-\alpha}$. As usual, let
$$Y( e^{\beta}, z) = \sum_{i \in \Z} e ^{\beta}_i z^{-i-1}, \quad \beta \in L$$

Define the Virasoro vector
$$\omega = \frac{1}{2 p} (\a (-1) ^2 + (p-2) \a (-2)){\bf 1}$$
and the (screening) operators
$Q = e ^{\a}_0$ and
$$G =\sum_{i =1 } ^{\infty} \frac{1}{i} e ^{\a}_{-i} e ^{\a}_i .$$

As shown in \cite{AdM-2009} the triplet vertex algebra $\WW_{2,p}$ can be realized as a subalgebra of $V_L$ generated by $\omega$ and primary vectors
$$ F = Q e ^{-3 \a}, \quad H = G F, \quad E = G ^2 F.$$

The results from \cite{AdM-2009} give that the results from Section \ref{hom-zhu} can be applied on $\WW_{2,p}$. In particular $\WW_{2,p}$ is ${\Z}$--graded and charge zero component is the singlet vertex algebra $\mip$ generated by $\omega$ and $H$.
We also have homomorphism $\Phi : A(\mip) \rightarrow A(\WW_{2,p})$.

We shall now consider the case $p=3$. By using analogous approach as in the case of $\triplet$ we prove the following result.

\begin{proposition} \label{ker-c0}
$\mbox{Ker}(\Phi)$ is contained in the following ideal in $A(\mip)$
$$ A(\mip) . \{ p([\omega]) * [H], f_{2,3}([\omega]) \} = \{ A([\omega]) * p([\omega]) * [H] + B([\omega]) * f_{2,3}([\omega]), \ \ A, B \in {\C}[x] \}, $$
where
 \bea  p(x) =  &&  (x-5) (x-7) (x-\tfrac{10}{3})
(x-\tfrac{33}{8})
  (x-\tfrac{21}{8}) (x-\tfrac{35}{24}) \nonumber \\
 f_{2,3}(x)=&& x ^3  \left( (x-1) (x-2) (x
-\tfrac{1}{8}) (x-\tfrac{5}{8}) (x-\tfrac{1}{3}) \right) ^2
\nonumber \\ && (x-5) (x-7) (x-\tfrac{10}{3})(x+\tfrac{1}{24})
(x-\tfrac{33}{8})
  (x-\tfrac{21}{8}) (x-\tfrac{35}{24}) . \nonumber \eea

\end{proposition}

\begin{corollary}
  The center of Zhu's algebra $A(\WW_{2,3})$ is $20$--dimensional and  it is  isomorphic to $${\C}[x] / \la f_{2,3}(x) \ra. $$
\end{corollary}
\begin{proof}
By using results from \cite{AdM-2009} one can easily see that the center of $A(\WW_{2,3})$ is isomorphic to the subalgebra generated by $[\omega]$.
By using Proposition \ref{ker-c0} we see that
$$ f \in {\C}[x], \ f([\omega]) \in \mbox{Ker} (\Phi) \quad \implies  f_{2,3} \vert f. $$
Since $f_{2,3} ([\omega]) \in \mbox{Ker}(\Phi)$, we prove the assertion.
\end{proof}

\begin{remark}
This result implies that $A(\WW_{2,3})$ has  $2$--dimensional indecomposable modules $U^{(2)} _{h}$ on which $[\omega]$ acts (in some basis) as
$$ \left(
   \begin{array}{cc}
     h & 1 \\
     0 & h \\
   \end{array}
 \right) $$
  where $h \in \{ 0, 1, 2, 1/8, 5/ 8, 1/3 \} $, and $3$--dimensional indecomposable module $U^{(3)} _{0}$ on which $[\omega]$ acts as
 $$
 \left(
   \begin{array}{ccc}
     0 & 1 & 0 \\
     0 & 0 & 1 \\
     0 & 0 & 0 \\
   \end{array}
 \right). $$
 Other zeros of the polynomial  $f_{2,3}$ do not give rise to indecomposable modules.

 By applying the theory of Zhu's algebras, we get the existence of   logarithmic, $\N$--graded   $\WW_{2,3}$--modules $R_h ^{(2)}$ and $R_0 ^{(3)}$ whose top components are isomorphic to $U_h ^{(2)}$ and $U_0 ^{(3)}$ respectively. These modules   appeared in the fusion rules analysis in \cite{GRW}.
\end{remark}
\vskip 5mm

By using representation theory of the vertex operator algebra $\WW_{2,3}$ we can conclude that
\bea && 0=[ F\circ E] =  g([\omega]) *  p([\omega]) * [H],  \label{rel-2-3} \eea
for certain polynomial $g$ of degree $2$.
We can determine polynomial $g$ by evaluating (\ref{rel-2-3}) on top components of $\overline{M(1)}$--modules.

By using direct calculation (and  Mathematica) we get:

$$ g(x) = \nu (\frac{62128128}{14003665} x ^2  - \frac{918683648}{14003665} x + \frac{5767168}{215441}) \quad (\nu \ne 0).$$
Moreover,  $g$ is relatively prime with $f_{2,3}$. Therefore, $\Phi(g([\omega])) $ is invertible in $A(\WW_{2,3})$ and we have that
in $p([\omega] ) * [H]  \in \mbox{Ker}(\Phi)$. This implies
$$\mbox{Ker}(\Phi) = A(\mip) . \{ p([\omega]) * [H], f_{2,3}([\omega]) \}. $$

Moreover, by simple analysis of Zhu's algebra we get that
$$ A(\WW_{2,3}) = A_{-1}( \WW_{2,3}) \oplus A_0( \WW_{2,3}) \oplus A_1( \WW_{2,3})$$
where
$$A_{0}( \WW_{2,3}) =  A(\mip) / \mbox{Ker}(\Phi) , \quad \dim A_{0}( \WW_{2,3}) =26,$$
$$A_{1}( \WW_{2,3})= A_{0}( \WW_{2,3}). [E], \quad \dim A_{1}( \WW_{2,3}) = 6, $$
$$A_{-1}( \WW_{2,3})= A_{0}( \WW_{2,3}). [F], \quad \dim A_{-1}( \WW_{2,3}) = 6. $$

More precisely we have that
$$ p([\omega]) * [E] = p([\omega]) * [F]=0,$$
$$ A_{1}( \WW_{2,3}) = \mbox{span}_{\C} \{ p_h([\omega]) * [E], h \in S_{2,3} ^{(2)} \}, $$
$$ A_{-1}( \WW_{2,3}) = \mbox{span}_{\C} \{ p_h([\omega]) * [F], h \in S_{2,3} ^{(2)} \}, $$
where
$$S_{2,3} ^{(2)} = \{ 5, 7, 10/3, 33/8, 21/8, 35/ 24 \}, \quad p(x)=(x-h)p_h(x). $$

This implies the following result:
\begin{proposition} We have
$$ \dim A( \WW_{2,3}) = 38. $$
\end{proposition}

\begin{remark}
It is interesting to notice that in this case we most likely have  $$ \dim A( \WW_{2,3}) < \dim {\mathcal P}( \WW_{2,3}). $$
\end{remark}

\begin{remark}   In \cite{AdM-2010}, among other things, we classified irreducible $\WW_{2,p}$--modules. This can be now used
to generalize results from this section for general $p$.
\end{remark}

\section{Zhu's algebra $A(\striplet)$ }
In this section we shall describe Zhu's algebra of the triplet vertex operator superalgebra $\striplet$ introduced in \cite{AdM-striplet}. In \cite{AdM-striplet} we classify irreducible $\striplet$--modules and proved that $\striplet$ is $C_2$--cofinite vertex operator superalgebra. There we also presented the conjecture that Zhu's algebra $A(\striplet)$ is $(6m +1)$--dimensional. In this section we shall prove this conjecture by using new methods from previous sections.

We should say that $\striplet$ can be considered as a super-analog of triplet vertex algebra $\triplet$, but $\striplet$ requires different techniques based on the representation theory of the  $N=1$ Neveu-Schwarz  Lie algebra.

We shall first recall definition of the triplet vertex superalgebra $\striplet$.

Let $V_L$ be the lattice vertex superalgebra associated to the lattice
$L= {\Z} \alpha$, with $\la \a , \a \ra = 2m+1$, $m \in \mathbb{N}$.

Let  $CL$ be the Clifford algebra, generated by
$\{ \phi  (n) , n \in  {\hf} + {\Z} \} \cup \{1\}$ and relations
$$ \{ \phi  ({n}) , \phi (m)\}=   \delta_{n,-m},
\quad n,m \in {\hf}+\mathbb{Z} .$$
Here $1$ is central.

Let $F$ be the $CL $--module generated by the vector ${\vak}$ such
that
$$ \phi ({n}) {\vak} = 0, \ n > 0.$$
 Then the field
$$Y( \phi ({-\hf}) {\bf 1} ,z) =\phi (z) = \sum_{ n \in {\hf} +
{\Z} } \phi ({ n }) z ^{-n- \hf}, $$
generates the unique vertex operator superalgebra structure on $F$.

 We define:
\bea &&\tau = \frac{1}{ \sqrt{2 m +1 }} \left( \alpha(-1) {\vak}
\otimes \phi (-\hf) {\vak} + 2 m {\vak} \otimes  \phi ({-\thf})
{\vak}
\right) , \nonumber  \\
&& G(z) = Y(\tau,z) = \sum_{n \in {\Z} } G(n+\hf) z ^{-n-2}, \nonumber \\
&& \omega = {\hf} G(-\hf) \tau, \quad  L(z) = Y(\omega,z) =\sum_{n \in {\Z} } L(n) z ^{-n-2}. \nonumber \eea

The components of the fields $L(z)$, $G(z)$ define on $V_L \otimes F$ a representation of the $N=1$ Neveu-Schwarz superalgebra $\NS$ with central charge
$c_{2m+1,1}=\frac{3}{2}(1-\frac{8m^2}{2m+1})$. Moreover, the operator
$$Q = \mbox{Res}_z Y( e ^{\alpha} \otimes \phi(-\tfrac{1}{2}) {\bf 1}, z)$$ is a screening operator which commutes with the action of the Neveu-Schwarz algebra.

 The $N=1$ vertex operator superalgebra $\striplet$ is defined to be a subalgebra
of $V_L \otimes F$ generated by superconformal vector
$\tau$ and
$$ F = e^{-\alpha}, \ H = Q e^{-\alpha}, \ E = Q ^2 e^{-\alpha},
$$
where
these three vectors are highest weight vectors for the
Neveu-Schwarz algebra $\NS$.

For $X \in \{E, F, H \}$, we define
$ \widehat{X} = G(-1/2) X. $

The representation theory of $\striplet$ was developed in \cite{AdM-striplet}. Recall that $\striplet$ is a simple, $C_2$--cofinite vertex superalgebra with $2m+1$ irreducible representations
$$ S \Lambda(1), \dots, S \Lambda(m+1), \ S\Pi(1), \cdots, S \Pi(m). $$

For $ i \in {\Z}$, we let  $$h ^{2 i +1,1} = \frac{(2m +1 - 2 i-1) ^2 - 4 m  ^2}{ 8 (2 m+1)} = \frac{(m - i) ^2 -  m  ^2}{ 2 ( 2 m+1)} .$$
 The
top component of  $S\Lambda(i+1)$ is $1$-dimensional and has conformal weight $h^{2i + 1,1}$,  and the top component of $S\Pi(i)$ is $2$--dimensional and has conformal weight
$h^ {2 (3m +1 -i) +1,1}$.

\vskip 5mm
The vertex subalgebra of $\striplet$ generated by $\tau$ and $H$ is called singlet vertex algebra and will be denoted by $\smip$.

\begin{theorem}  \cite{AdM-striplet} \label{zhu-alg-n1}
Zhu's algebra $A(\overline{SM(1)})$   is
    isomorphic to the
commutative algebra \\  ${\C}[ x, y] / \la P(x,y) \ra$,
where  $\la P(x,y) \ra $ is the principal ideal generated by
$$ P(x,y) = y ^{2} -C_m \prod_{i=0} ^{2m} (x- h^{2i+1,1})  $$
where
$C_m = \frac{2 ^{2m+1} (2m+1) ^{2m+1} }{(2m+1)!}. $
(Here
$x$ and $y$ correspond to $[\omega]$ and $[\widehat{H}]$.)
\end{theorem}

 Recall that as a module over Neveu-Schwarz  algebra $\striplet$ is generated by singular vectors
 $$ \{ Q ^j  e ^{-n \a}, \ n \in {\N}, 0 \le j \le 2n \}, $$
 and $\smip$ is generated by singular vectors
  $$ \{ Q ^n  e ^{-n \a}, \ n \in {\N} \}. $$

For every $ \ell \in {\N}$, we define
$$ \striplet_{-\ell} = \smip . e ^{-\ell \a} = \bigoplus_{n = 0} ^{\infty} U(\NS) . Q ^n e ^{-(n+\ell) \a },  $$
$$ \striplet_{\ell} = \smip . Q ^{ 2 \ell}  e ^{-\ell \a} = \bigoplus_{n = 0} ^{\infty} U(\NS) . Q ^{n + 2 \ell} e ^{-(n+\ell) \a}  . $$
\begin{proposition}
For every $\ell \in {\Z}$, $\striplet_{\ell}$ is an irreducible $\smip$--module and
$$ \striplet = \bigoplus_{\ell \in {\Z}} \striplet _{\ell}.$$
Moreover, for $v \in \striplet _{\ell_1}$, $w \in \striplet_{\ell_2}$, we have
$$Y(v,z ) w \in \striplet_{\ell_1 + \ell_2} ((z)).$$
Operator $G= Q$ satisfies conditions (\ref{uv-1})-(\ref{uv-5}).
\end{proposition}

The proof of the following two lemmas is completely analogous as in the case of triplet vertex algebra $\triplet$.
\begin{lemma}
 We have
 $$O(\striplet)_0 \subset \bigoplus_{n=1} ^{\infty} U(\NS). Q^n e ^{-n\a}  + O(\smip). $$
 \end{lemma}

Define now two polynomials
$$\ell(x) = \prod_{i=2m +1 } ^{3m } (x- h^{2 i+1,1}), \quad Sf_m(x) = \prod_{i=0} ^{3 m} (x- h^{2 i +1 ,1}).$$

\begin{lemma} \label{lema-2-n1}
Assume that  $n \ge 2$. Then in Zhu's algebra $A(\smip)$ we have
$$[ g \  Q ^n e ^{-n \a}]   = A([\omega]) * [\widehat{H}]  + B([\omega]) * Sf_m ([\omega]), $$
where
   $A, B \in {\C}[x]$, $ g \in U (\NS)$.
\end{lemma}

Now we shall consider the homomorphism
$$ \Phi : A(\smip) \rightarrow A(\striplet) $$ from Section \ref{hom-zhu}.

We need the following result:

\begin{lemma}
Inside $A(\smip)$ we have
$$ [E \circ F] = D_m \ell([\omega]) * [\widehat{H}] = 0, \quad \ D_m \ne 0. $$
\end{lemma}
\begin{proof}
We consider $[E\circ F]$ as an element of $A(\smip)$. Therefore
$$
 [E \circ F] = f ([\omega]) * [\widehat{H}] + g([\omega]) $$
 for certain polynomials $f, g \in {\C}[x]$. Then we shall evaluate both sides of this equality on a family of  $A(\smip)$--modules.
 Then  Theorem  \ref{identity-striplet} proven below (see also Proposition 8.1 and  Appendix  from  \cite{AdM-striplet}) gives
 $$- { 2m \choose m} ^2 { t + m \choose 4 m +1} = f\left( \frac{t (t-2m) }{2 (2m+1)} \right) { t \choose 2m+1} + g \left( \frac{t (t-2m) }{2 (2m+1)} \right) $$
 for arbitrary $t \in {\C}$. This easily gives that $g(x)=0$ and
 $f(x) = D_m \ell(x)$ for certain non-vanishing constant $D_m$.
\end{proof}

The previous result shows that
$$ \ell([\omega]) * [\widehat{H}] \in \mbox{Ker}(\Phi). $$
The proof of the following theorem is now   completely analogous to that of Theorem \ref{jezgra}.

\begin{theorem} \label{jezgra-n1} We have
\bea \mbox{Ker} ( \Phi) &=&A(\smip) .( \ell ([\omega]) * [ \widehat{H}] ) \nonumber \\ &\cong&  \mbox{span} _{\C} \{ A([\omega]) *\ell ([\omega])* [\widehat{H}] + B([\omega]) * Sf_m([\omega]), \quad A, B \in {\C}[x] \}. \nonumber \eea
\end{theorem}

\vskip 5mm

Let $$A_0 (\striplet) = \mbox{Im} (\Phi) \cong A(\smip) / \mbox{Ker} ( \Phi). $$

By combining the results from Section 11 of  \cite{AdM-striplet} and the above results we obtain the description of Zhu's algebra $A(\striplet)$. We get
$$ A(\striplet) = A_{-1} (\striplet) \oplus A_0 (\striplet) \oplus A_1 (\striplet),$$
where $$ A_{-1} (\striplet) = A_0 (\striplet) . [\widehat{F}], \quad   A_{1} (\striplet) = A_0 (\striplet) . [\widehat{E}], $$
and
$$ \dim A_{\pm 1} (\striplet) =m, \quad \dim A_0( \striplet) \le 4m +1 .$$
More precisely,
we have
$$ \ell ([\omega]) * [\widehat{E}] = \ell ([\omega]) * [\widehat{F}]=0,$$
$$ A_{1} (\striplet) = \mbox{span}_{\C} \{ \ell_i ([\omega]) * [\widehat{E}], \ \ 2m+1 \le i \le 3m+1 \},$$ $$A_{-1} (\striplet) = \mbox{span}_{\C} \{ \ell_i ([\omega]) * [\widehat{F}], \ \ 2m+1 \le i \le 3m+1 \},$$
where $\ell(x) =\ell_i(x)(x-h ^{2i+1,1})$.
Now Theorem \ref{jezgra} implies  that $\dim A_0 (\triplet) = 4 m+1. $
Therefore $\dim A(\striplet) = 6 m +1$. In this way we have proved the following theorem.

\begin{theorem} We have
$$ \dim A(\striplet) = 6 m +1. $$
\end{theorem}
This theorem was conjectured in \cite{AdM-striplet}. We now have completed description of the Zhu algebra $A(\striplet)$.

\begin{theorem} Zhu's algebra $A(\striplet)$ decomposes as a sum of ideals
$$A(\striplet)=\bigoplus_{i=2m+1}^{3m}  \mathbb{M}_{h^{2i+1,1}} \oplus \bigoplus_{i=0}^{m-1} \mathbb{I}_{h^{2i+1,1}}
\oplus \mathbb{C}_{h^{2m+1,1}},$$ where $\mathbb{M}_{h^{2 i+1,1}}
\cong M_2(\mathbb{C})$, $\dim(\mathbb{I}_{h^{2
i+1,1}}) =2$ and $\mathbb{C}_{h^{2m+1,1}}$ is one-dimensional.
\end{theorem}

The ideals $\mathbb{I}_{h^{2
i+1,1}}$ give a family of $2$--dimensional indecomposable modules for $A(\striplet)$.

By applying  Zhu's correspondence we get:

\begin{corollary}
For every $1 \le i \le m$, there exits the logarithmic, self-dual,  $\N$--graded $\striplet$--module $\mathcal{SP}_{i} ^+$ such that the top component $\mathcal{SP}_{i} ^+ (0)$ is two-dimensional and
$L(0)$ acts on   $\mathcal{SP}_{i} ^+ (0)$ (in some basis) as
$$\left(
    \begin{array}{cc}
      h^{2 i+1,1} & 1 \\
      0 & h^{2 i +1,1} \\
    \end{array}
  \right) .$$

\end{corollary}

 \section{Twisted Zhu's algebra $A_{\sigma} (\striplet)$}

Every vertex operator superalgebra $V ^{\bar 0} \oplus V ^{\bar 1}$ has the canonical
parity automorphism $\sigma$, where $\sigma_{V^{\bar 0}}=1$ and
$\sigma_{V^{\bar 1}}=-1$.  We briefly recall the notion of $\sigma$-twisted Zhu's algebra (cf. \cite{Xu}).

Consider the subspace $O_{\sigma}(V)
\subset V$, spanned by elements of the form
$${\rm Res}_x \frac{(1+x)^{{\wt}(u)}}{x^2} Y(u,x)v,$$
where $u \in V$ is homogeneous.
It can be easily shown that
$${\rm Res}_x \frac{(1+x)^{{\wt }(u)}}{x^n} Y(u,x)v \in O_{\sigma} (V) \ \mbox{for} \  \ n \geq 2.$$

Then, the vector space $A_\sigma(V)=V/O_{\sigma}(V)$ is equipped with
an associative algebra structure via
$$u * v={\rm Res}_x \frac{(1+x)^{\rm deg(u)}}{x} Y(u,x)v$$
 An important difference between the untwisted
associative algebra and $A_\sigma(V)$ is that $A_\sigma(V)$ is
$\mathbb{Z}_2$--graded, so
$$A_\sigma(V)=A^0_\sigma(V) \oplus A^1_\sigma(V).$$
  We shall often use $[a] \in A_{\sigma}(V)$ for
the image of $a \in V$ under the map $V \longrightarrow
A_{\sigma}(V)$.

It is not hard to prove the following result.

\begin{proposition} Let $V$ be a vertex operator superalgebra as in Section 2.
There is a natural surjective superalgebra map from $\mathcal{P}(V)$ to $gr(A_{\sigma}(V))$.
\end{proposition}

Here we shall consider the $N=1$ vertex operator superalgebras $ \smip$ and $\striplet$ and the corresponding twisted Zhu's algebras $A_{\sigma} (\smip)$ and $A_{\sigma} (\striplet)$.

\begin{theorem} \cite{AdM-tstriplet}
The associative algebra $A_{\sigma}(\overline{SM(1)})$ is isomorphic to the
${\Z}_2$--graded commutative associative algebra $${\C}[x,y] / \la
H(x,y) \ra$$ where $\la H(x,y) \ra$ is (two-sided) ideal in
${\C}[x,y]$, generated by
$$H(x,y) = y^2 - \widetilde{C}_m \prod_{i=0}^{m-1}\left(x^2-\frac{(2i+1-2m)^2}{8(2m+1)}\right)^2 ,$$
where  $ \widetilde{C}_m = \frac{2^{2m-1} (2m+1) ^{2m} }{ (2m)! ^2}. $
(Here $x$ and $y$ correspond to $[\tau]$ and $[H]$).
\end{theorem}

Define $$ h ^{ 2 i +2,1} = \frac{  (2 m +1 - 2 i -2)^2 - 4 m ^2}{ 8 (2 m+1)} + \frac{1}{16}. $$
We also have a natural homomorphism $\Phi : A_{\sigma} (\smip) \rightarrow A_{\sigma} (\striplet)$. By using similar approach as in the case of untwisted Zhu's algebras $A(\triplet)$ and $A(\striplet)$ we get the following result.

\begin{proposition} \label{ker-ramond}
$\mbox{Ker}(\Phi)$ is contained in the following ideal in $A_{\sigma}(\smip)$
$$ A_{\sigma}(\smip) . \{ r([\omega]) * [ H ] \} = \{ A([\tau]) * r ([\omega]) *   [ H ] + B([\tau]) * Rf_{m}([\omega]), \ \ A, B \in {\C}[x] \}, $$
where
 \bea
[\omega] =  &&  [\tau] ^2 + \frac{c_{2m+1,1}}{24}, \nonumber \\
 r(x) =  &&  \prod_{i= 2m } ^{3m} (x-h ^{2 i +2,1}), \nonumber \\
 Rf_{m}(x)=&&  \prod_{i= 0} ^{3m} (x-h ^{2 i +2,1}) . \nonumber \eea

\end{proposition}

\begin{corollary}
  The center of Zhu's algebra $A_{\sigma}(\striplet)$ is $ 6 m +2$--dimensional and  it is  isomorphic to $${\C}[x] / \la Rf_m( x ^2 +\tfrac{c_{2m+1,1}}{24} ) \ra. $$
\end{corollary}

Define the following two vectors in $\smip$:
$$  U^{F,\widehat{E}} =\mbox{Res}_{z} Y(F,z) \widehat{E} \frac{(  1+ z) ^{ 2m + \tfrac{1}{2} }} {z ^2},  \quad   U^{F,E} =\mbox{Res}_{z} Y(F,z) E \frac{(  1+ z) ^{ 2m + \tfrac{1}{2} }} {z ^3}.  $$

 \begin{lemma} \label{lem-tw-1}
 In $A_{\sigma} (\smip)$ we have
 $$  [U^{F,\widehat{E}}] = A_m   \  r([\omega])* [H], \quad A_m \ne 0. $$
 \end{lemma}
\begin{proof}
Let $v_{\lambda} $ be the highest weight vector in $\sigma$--twisted $\smip$--module $M(1,\lambda) \otimes M  $, where $M(1,\lambda)$ is an irreducible module for the Heisenberg vertex algebra $M(1)$ and $M$ is the $\sigma$--twisted module for Clifford vertex superalgebra $F$ (see  \cite{AdM-tstriplet} for details). By direct calculation which uses the concepts from \cite{AdM-tstriplet} we get
$$o(U ^{F,\widehat{E}}) v_{\lambda}  = (R^{(1)} _m (t) + R ^{(2)} _m (t) ) v_{\lambda} $$
where
\bea
R^{(1)} _m (t) & =  & \nu_1
Res_{z_1,z_2,z_3} \frac{(1+z_1)^{2m+1/2-t}}{z_1^2} (z_1 z_2 z_3)^{-2m-1} \nonumber \\
&& \cdot (1-z_2/z_1)^{-2m-1}(1-z_3/z_1)^{-2m-1}(z_2-z_3)^{2m}
(1+z_2)^{t+1/2}(1+z_3)^{t-1/2}, \nonumber \\
&& \ \nonumber \\
R^{(2)} _m (t) & = & \nu_2
Res_{z_1,z_2,z_3} \frac{(1+z_1)^{2m+1/2-t}}{z_1} (z_1 z_2 z_3)^{-2m-2} \nonumber \\
&& \cdot (1-z_2/z_1)^{-2m-1}(1-z_3/z_1)^{-2m-1}(z_2-z_3)^{2m+2}
(1+z_2)^{t- 1/2}(1+z_3)^{t-1/2} \nonumber
\eea
and $\nu_1, \nu_2$ are non-zero complex numbers.
Now Theorem  \ref{tw-nova} gives that
$$ R^{(1)} _m (t) =0, \qquad R^{(2)} _m (t) = {\nu} { t + m + 1/2 \choose 4 m +2} \qquad (\nu \ne 0). $$
This easily implies that in the  twisted Zhu's algebra $A_{\sigma} (\smip)$ we have
$$[U ^{ F,\widehat{E}}] = A_m  r([\omega]) * [H], \quad ( A _m \ne 0).$$
The proof follows.
\end{proof}

 \begin{lemma} \label{lem-tw-2}
 Inside $A_{\sigma} (\smip)$ we have
 $$  [U^{F,E}] = B_m   \  r([\omega])* [\tau] * [H], \quad B_m \ne 0. $$
 \end{lemma}
 \begin{proof}
 Everything here is a matter of rewriting the residue explicitly. The rest follows from Theorem \ref{identity-tstriplet}.
  \end{proof}

By using Lemma \ref{lem-tw-1}  (or using Lemma \ref{lem-tw-2} and the fact that $[\tau]$ is a unit in $A_{\sigma}(\striplet)$), we have
 $$r([\omega]) * [H] \in \mbox{Ker}(\Phi). $$

 Therefore,
 \bea && \label{rel-jezgra-t} \mbox{Ker} (\Phi) = A_{\sigma} (\smip) . \{ r([\omega])  * [H] \}. \eea
 Now we are in position to describe Zhu's algebra of $A_{\sigma}(\striplet)$. By using Proposition 6.2 of \cite{AdM-tstriplet} we get:
  $$ A_{\sigma} (\striplet) = A_{\sigma} (\striplet)_{-1} \oplus A_{\sigma} (\striplet) _0  \oplus A_{\sigma} (\striplet)_{1},$$
where
$$ A_{\sigma} (\striplet) _0 = A_{\sigma}(\smip) / \mbox{Ker}(\Phi), $$
$$ A_{\sigma} (\striplet)_{-1} = A_{\sigma} (\striplet) _0 . [F], \quad   A_{\sigma} (\striplet)_{1} = A_{\sigma} (\striplet) _0 . [E]. $$
 \vskip 5mm

 Relation (\ref{rel-jezgra-t}) gives that
 $ \dim A_{\sigma} (\striplet) _0 = 8 m + 4.$ One can also see that
 $$ A_{\sigma} (\striplet)_{1} = \mbox{span}_{\C} \{ r^{\varepsilon}_i([\tau] ) * [E], \ 2m \le i \le 3m , \ \varepsilon = \pm \}, $$
 $$ A_{\sigma} (\striplet)_{-1} = \mbox{span}_{\C} \{ r^{\varepsilon}_i([\tau] ) * [F], \ 2m \le i \le 3m, \ \varepsilon = \pm \}, $$
 where
 $$
r (x ^2 +\tfrac{c_{2m+1,1}}{24} )= r^{\pm}_i(x ) ( x \pm \frac{ 2m - 2 i -1}{\sqrt{8 (2m +1)}}). $$
 Therefore:
 $$ \dim A_{\sigma} (\striplet)_{\pm 1} = 2m +2 . $$
In this way we have proved the following result (conjectured in \cite{AdM-tstriplet}):
 \begin{theorem} We have
 $$\dim A_{\sigma}(\striplet) = 12 m + 8. $$
 \end{theorem}

\section{The $C_2$--algebra $\mathcal{P}(\striplet)$ }

Now we are in the position to determine $\mathcal{P}(\striplet)$. Firstly, observe that
$$ \dim \mathcal{P}(\striplet) \ge  \dim A_{\sigma}(\striplet) = 12 m + 8. $$
Therefore we only have to prove that
$$  \dim \mathcal{P}(\striplet) \le 12 m + 8. $$
By using results from \cite{AdM-striplet}
we conclude that
  $\mathcal{P}(\striplet)$  is generated by
$$ \overline{\tau}, \overline{\omega}, \overline{E}, \overline{F}, \overline{H}, \overline{\widehat{E}}, \overline{\widehat{F}},  \overline{\widehat{H}}. $$
Also the following relations hold:
$$ \overline{\tau } ^2 = \overline{\omega} ^{3 m +1} = \overline{\widehat{E}} ^2 = \overline{\widehat{F}} ^2= \overline{\widehat{H}} \cdot \overline{\widehat{E}}  = \overline{\widehat{H}} \cdot \overline{\widehat{F}}  =   0,  $$
$$ \overline{X} ^2 = \overline{\tau} \overline{\widehat{X}} = 0,    \overline{\tau} \overline{X} = {\nu}_1 \overline{\widehat{X}}, \ \overline{\widehat{H}} ^2 = {\nu}_2 \overline{\omega} ^{2m+1}, $$
where ${\nu}_1, \nu _2$ are non-zero complex numbers and $X \in \{E, F, H \}$.

Therefore every element $u \in \mathcal{P}(\striplet)$ has the form
\bea
u = && f_1 (\overline{\omega}) + f_2 (\overline{\omega}) \overline{E} + f_3 (\overline{\omega}) \overline{F} + f_4 (\overline{\omega}) \overline{H} + g_1 (\overline{\omega}) \overline{\tau} + g_2 (\overline{\omega}) \overline{\widehat{E}} + g_3 (\overline{\omega}) \overline{\widehat{F}} + g_4 (\overline{\omega}) \overline{\widehat{H}}, \nonumber \eea
for certain polynomials $f_i, g_i \in {\C}[x]$, \ $\deg(f_i), \deg(g_i) \le 3m$, \  $i =1, \cdots, 4$.

By using  Lemma \ref{lem-tw-1} and  Lemma\ \ref{lem-tw-2} we get:
\begin{proposition} \label{str-c2} We have
\bea &&F_{-2}  \widehat{E}\equiv   A_m  L(-2) ^{ m+1} H   + v_1 \quad \mbox{mod} ( C_2( \striplet) ) \qquad (A_m \ne 0), \label{rel-tw-1} \\
&&F_{-3}   {E} \equiv   B_m L(-2) ^{ m+1} \widehat{H } + v_2 \quad \mbox{mod} ( C_2( \striplet) ) \qquad (B_m \ne 0), \label{rel-tw-2} \eea
where $v_1, v_2 \in U(\NS). {\bf 1}$, $\wt (v_1) = 4 m + 5/2$, $\wt (v_2) = 4m +3$.
\end{proposition}
\begin{proof}
Let us prove relation (\ref{rel-tw-1}).   First we noticed  $F_{-2} \widehat{E}$  is an odd vector in $\smip$ of conformal weight $4m+5/2$. This easily implies that $F_{-2} \widehat{E}$ has the form
 $$ F_{-2}  \widehat{E} =   A  L(-2) ^{ m+1} H  + v' $$
 for certain $A \in \C$ and $v' \in  C_2( \striplet) )$.
 Then relation in Zhu's algebra $A_{\sigma}(\smip)$  from Lemma \ref{lem-tw-1} gives that $A= A_m \ne 0$.

Relation  (\ref{rel-tw-2}) follows  from  Lemma \ref{lem-tw-2} in the same way.
\end{proof}

By using Proposition \ref{str-c2} and the action of the operator $Q$ we get:
$$\overline{\omega} ^{m+1} \overline{X}   =   \overline{\omega} ^{m+1} \overline{\widehat{X}} = 0 \  \quad X \in \{E, F, H\},$$
\vskip 5mm
The analysis above   implies that $\dim \mathcal{P}(\striplet) \le 12 m + 8.$ In this way we have proved the following result:

\begin{theorem} We have
$$\dim \mathcal{P}(\striplet) = \dim A_{\sigma}(\striplet).$$
\end{theorem}

From the description of $\mathcal{P}(\striplet)$ we see ${\rm dim} \ \mathcal{P}_0(\striplet) =(3m+1)+3(m+1)=6m+4$.
Consequently:
\begin{corollary} For all $m \in \mathbb{N}$,
\be
\dim A(\striplet) < \dim \mathcal{P}_0(\striplet).
\ee
\end{corollary}

\section{Constant term identities}

In this section, which is mostly of  combinatorial nature, we obtain constant (or residue) term identities needed in the paper.
We are interested in certain multiple sums, which after several steps reduce to a single sum. Although we only need non-vanishing
condition for these sums, we in fact provide closed evaluation expression which is of independent interest. The main tool is
Wilf-Zeilberger (WZ) theory of summation elaborated in more details in the appendix.

As usual, for $t \in \C$ and $k \in \N$ we set
$${ t \choose k} = \frac{t (t-1) \cdots (t-k+1)}{ k!}. $$
Define also
$$H_p(t)={2p \choose p}{2p-2 \choose p-1}{t+p \choose 4p-1}.$$

\begin{theorem} \label{identity-triplet}
Let

\be \label{known}
G_p(t)=o(E \circ F)v_{\lambda}=\sum_{i \geq 0} {2p-1 \choose i}  o(E_{-2+i} F ) v_{\lambda}=G_p(t)v_{\lambda},
\ee
where we view $t$ as a formal variable.
Then

$$G_p(t)={\rm Res}_{x_1,x_2,x_3} \frac{1}{(x_1 x_2 x_3)^{2p}} \frac{(x_2-x_3)^{2p}}{(1-x_2/x_1)^{2p}(1-x_3/x_1)^{2p}}\frac{(1+x_1)^{2p-1-t}}{x_1^2}(1+x_2)^{t}(1+x_3)^{t},$$
and
\be \label{more-imp}
G_p(t)=H_p(t).
\ee
\end{theorem}

{\em Proof.}

Easy inspections shows that $G_p(t)$ is a polynomial in $t$.

Formula (\ref{known}) has already been proven in \cite{AdM-triplet}, so we only focus on (\ref{more-imp}).
The proof is divided into 5 steps:

In Step 1 we prove $G_p(t)$ vanishes at $t \in \{0,\dots ,2p-1 \}$.

In Step 2 we prove  $G(t)=(-1)^p G(2p-2-t)$ (the skew-symmetry).

In Step 3 we prove  $G_p(t)$ vanishes at $t \in \{ -p,\dots ,-1 \}$.

In Step 4 we apply skew-symmetry to show that $G_p(t)$ vanishes at $ t \in \{ 3p-2, \dots ,2p \}$.
Steps 1-4 imply that $G_p(t)$ is divisible by ${t+p \choose 4p-1}$.

In Step 5 we show that
$$H'(0)=G'(0) \neq 0,$$
that is $H(t)$ and $G(t)$ have the same first derivative at zero.

{\em Step 1.} This is an easy observation, which follows simply from
consideration of $Res_{x_1} G_p(t)$. For every $t=i \in \{1,...,2p-1 \}$ the highest positive powers of $x_1$ appearing in the Laurent expansion of $G_p(t)$ is at most ${2p-1-i}$, coming
from the expansion of $(1+x_1)^{2p-1-i}$. Contribution from other terms containing $x_1$ is  $x_1^{-2p-2-j}$, where $j \geq 0$. Thus $G_p(i)=0$. If $t=i=0$ one easily sees that the constant term is zero.

{\em Step 2.} Apply the substitution $y_i=\frac{x_i}{1+x_i}$ and the general  formula
$${\rm Res}_{y_i} F(y_i)={\rm Res}_{x_i} y'(x_i)F(y(x_i)),$$
where $y_i = x_i+\cdots \in x_i {\C}[[x_i]].$

{\em Step 3.} This part is more tricky because for  $t=-1,-2,..,-p$ the terms $(1+x_i)^t$ have an infinite power expansion in $x_i$. So let $-t=k \in \{1,...,p \}$. We clearly have
$$ \frac{1}{(x_1 x_2 x_3)^{2p}} \frac{((1+x_2)-(1+x_3))^{2p}}{(1-x_2/x_1)^{2p}(1-x_3/x_1)^{2p}}\frac{(1+x_1)^{2p-1+k}}{x_1^2}(1+x_2)^{-k}(1+x_3)^{-k},$$
\be \label{rat}
=\sum_{i+j=2p} (-1)^j {2p \choose i} \frac{1}{(x_1 x_2 x_3)^{2p} (1-x_2/x_1)^{2p}(1-x_3/x_1)^{2p}}\frac{(1+x_1)^{2p-1-t}}{x_1^2}\frac{(1+x_2)^i}{(1+x_2)^k}\frac{(1+x_2)^j}{(1+x_3)^{k}},
\ee

Now, for every $k$ in the range we either have  $$i \geq k, \ \ {\rm or} \ \ j \geq k,$$ (otherwise $2p=i+j < 2k$, contradicting our choice of $k$). If $i \geq k$ we will consider ${\rm Res}_{x_2}$ of (\ref{rat})
(if $j \geq k$ we consider ${\rm Res}_{x_3}$ instead and the argument follows verbatim). In the expression $(1+x_2)^{i-k}$ the highest power of $x_2$ is clearly $i-k$. In addition we already have $x_2^{-2p}$ contribution, so the highest power
of $x_2$ we get from these two terms is $x_2^{-2p+i-k}$. For  ${\rm Res}_{x_2}$ to be nontrivial we must have additional  $(2p+k-i-1)$ powers of $x_2$. This can come only from the expansion of $(1-x_2/x_1)^{-2p}$, meaning
that we need term $(x_2/x_1)^{2p+k-i-1}$ in its expansion (and higher powers). Now, we consider ${\rm Res}_{x_1}$. We already have in  (\ref{rat}) the factor $\frac{1}{x_2^{2p+2}}$, so we have $\frac{1}{x_1^{4p+2+k-i-1}}$
 as the term with the highest power of $x_1$. The highest positive power of $x_1$ is $(2p-1+k)$, which comes from expansion of $(1+x_1)^{2p-1+k}$. Since we are taking ${\rm Res}_{x_1}$ we see that
 $$2p-1+k-(4p+1+k-i)=-2p-2+i,$$
which is always $\leq -2$, thus gives the zero contribution to ${\rm Res}_{x_1}$, and $G_p(k)=0$.

{\em Step 4.} Since $G_p(t)=(-1)^p G_{p}(2p-2-t)$, Step 3 gives $G_p(i)=0$, for $i=2p,...,3p-2$.

{\em Step 5.} Here, it  is convenient to rewrite $G_p(t)$ as
\be \label{full-sum}
\sum_{(i,j,k)=(1,0,0)}^{(2p-1,2p-i-1,i-1)} {2p \choose i} {-2p \choose j}{-2p \choose k} {t \choose 2p-1-i-j}{t \choose i-1-k}{2p-1-t \choose 2p+j+k+1}.
\ee
We have to determine the linear coefficients in $G_p(t)$ (as we already know the constant term is zero). Observe that
$${2p-1-t \choose 2p+j+k+1} \in \lambda_1 t+ \cdots \in t \mathbb{C}[t],$$
where $\lambda_1 \neq 0$ for all $j$ and $k$!  Similarly, for $2p-1-i-j \neq 0$ and $i-1-k \neq 0$ we also
have  $${t \choose 2p-i-j-1} \in \nu_1 t + \cdots ,  \ \ \nu_1 \neq 0$$ and
$${t \choose i-1-k} \in \epsilon_1 t+ \cdots , \ \ \epsilon_1 \neq 0,$$
and trivially ${t \choose 0}=1$, if $2p-1-i-j=0$ or $i-1-k=0$. In conclusion, to extract the linear term from $G_p(t)$, it is sufficient to consider the case
\be \label{range}
2p-1-i-j=0, \ \ i-1-k=0.
\ee
With this choice
$${2p-1-t \choose 2p+j+k+1}=\frac{(2p-1)!(2p-1)!}{(4p-1)!}t + \cdots,$$
where dots denote the higher powers of $t$. For $j$ and $k$ subject to (\ref{range}), we have $j=2p-i-1$ and $k=i-1$, so we
 now have
$$G_p(t)= \frac{(2p-1)!^2}{(4p-1)!} t \left( \sum_{i=1}^{2p-1} (-1)^i {2p \choose i} {-2p \choose 2p-1-i} {-2p \choose i-1} \right) + \cdots,$$
where again the dots denote the higher order terms.
The sum in the parenthesis, denoted by $f(p)$, can be evaluated via Zeilberger's algorithm (cf. Appendix). We get
$$3 (3 p - 4) (2 p - 3) (3 p - 2) f(p - 1) + f(p) (2 p - 1) (p - 1)^2=0, \ \ p \geq 2$$
The last formula yields  (after iteration) $$f(p)=(-1)^p \frac{2(3p-2)!}{(2p-1)(p-1)!^3},$$
so we finally have
$$G_p(t)= (-1)^p \frac{(2p-1)!^2}{(4p-1)!}\frac{2(3p-2)!}{(2p-1)(p-1)!^3} t + \cdots=(-1)^p \frac{(2p)! (2p-2)! (3p-2)!}{p! (p-1)!^2 (4p-1)!}t+ \cdots,$$
But this coefficient is precisely the linear coefficients of
$$H_p(t)=(-1)^p {2p \choose p} {2p-2 \choose p-1} \frac{(3p-2)! p!}{(4p-1)!} t +\cdots,$$
and $H'(0)=G'(0)$.

\qed

\begin{remark}
1. It is tempting to ask whether WZ-theory  can be applied directly to (\ref{full-sum}).
The short answer seems to be "no", or at least we couldn't make it work even after we simplify (\ref{full-sum}) to a single sum involving
generalized hypergometric series.

2. There seems to be another degree of freedom in the formula for $f(p)$. We can for instance show that for $k \geq 1$, $p \geq k$:
$$ \sum_{i=k}^{2p-k} (-1)^i {2p \choose i} {-2p \choose 2p-k-i} {-2p \choose i-k} =\frac{2(-1)^p (3p-1-k)!}{ (2p-k) (p-1)!^2 (p-k)!}.$$
Our formula for $f(p)$ is obtained by specializing $k=1$.

\end{remark}



\begin{theorem} \label{identity-striplet}
The residue
$$\sum_{i=0}^{2m}  {\rm Res}_{z_1,z_2,z_3} \frac{(1+z_1)^{2m}}{z_1}z_2^{-i-1}z_3^i (z_1 z_2 z_3)^{-2m-1}$$
$$ \cdot (1-z_2/z_1)^{-2m-1}(1-z_3/z_1)^{-2m-1}(z_2-z_3)^{2m+1}
\frac{(1+z_2)^t(1+z_3)^t}{(1+z_1)^t}$$
equals
$$-{2m \choose m}^2 {t+m \choose 4m+1}.$$
\end{theorem}

\begin{proof}
The proof is analogous to the previous theorem so we omit some details. Let  $\tilde{H}_m(t)=-{2m \choose m}^2 {t+m \choose 4m+1}.$
Observe first that the sum  in question can be rewritten more compactly as
$${\rm Res}_{z_1,z_2,z_3} \frac{(1+z_1)^{2m}}{z_1} (z_1 z_2 z_3)^{-2m-1} \cdot$$
$$ \cdot (1-z_2/z_1)^{-2m-1}(1-z_3/z_1)^{-2m-1}(z_2-z_3)^{2m}
\frac{(1+z_2)^t(1+z_3)^t}{(1+z_1)^t},$$
which equals
\be \label{full-3sum}
\tilde{G}_m(t):=\sum_{(i,j,k)=(0,0,0)}^{(2m,2m-i,i)} (-1)^{i+j+k}{-2m-1 \choose k} {-2m-1 \choose j}{2m \choose i} {t \choose i-k}{t \choose 2m-j-i}{2m-t \choose 2m+1+k+j}.
\ee
As in the previous theorem we analyze the roots of $\tilde{G}_m(t)$ and show it is divisible by
$ {t+m \choose 4m+1}$.
Then we write
$$\tilde{H}_m(t)=-\frac{(2m)!^2}{(4m+1)!} t (\sum_{i=0}^{2m} (-1)^i {-2m-1 \choose i} {-2m-1 \choose 2m-i}{2m \choose i}+\cdots).$$
To prove
\be \label{super-rec}
\frac{(2m)!^2}{(4m+1)!}  \sum_{i=0}^{2m} (-1)^i {-2m-1 \choose i}{-2m-1 \choose 2m-i}{2m \choose i}=
{2m \choose m}^2 \frac{m! (-1)^m (3m)!}{(4m+1)!}.
\ee
we denote the sum on the left hand side in (\ref{super-rec}) by $Sum(m)$.  By using Zeilberger's algorithm
we obtain
$$-Sum(m+1)(m+1)^2-3(3m+1)(3m+2)Sum(m)=0,$$
so we get
\be \label{super-rec2}
Sum(m)= \sum_{i=0}^{2m} (-1)^i {-2m-1 \choose i}{-2m-1 \choose 2m-i}{2m \choose i}=\frac{(-1)^m (3m)!}{m!^3},
\ee
which yields the claim (see Appendix for more details).
To finish the proof we only have to argue
$$\tilde{H}'_m(0)=\tilde{G}'_m(0),$$
which is easy to check.
\end{proof}

\begin{theorem} \label{identity-tstriplet}
For $m \in \mathbb{N}$, we have
$${\rm Res}_{z_1,z_2,z_3} \frac{(1+z_1)^{2m+1/2-t}}{z_1 ^3} (z_1 z_2 z_3)^{-2m-1}$$
$$ \cdot (1-z_2/z_1)^{-2m-1}(1-z_3/z_1)^{-2m-1}(z_2-z_3)^{2m}
(1+z_2)^{t+1/2}(1+z_3)^{t-1/2}$$
equals
$$\frac{1}{(4m+3)(2m-1)}{2m \choose m}{2m+1 \choose m}(t-m){t+\frac{1}{2}+m \choose 4m+2}.$$
\end{theorem}

\begin{proof} Compared to the proof of Theorem \ref{identity-striplet} the strategy here is a bit different. We denote the residue by $F(m,t)$ and the product of binomial
coefficients by $G(m,t)$.

As in Step 1 of Theorem \ref{identity-triplet} we first identify trivial half-integer zeros. Then (as in Step 2)
we obtain the symmetry identity
$$F(m,t)=-F(m,2m-t).$$
The last relation also implies $F(m,m)=0$. Another application of the same formula gives the remaining zeros.

In the last step we shall argue $G'(m,t_0)=F'(m,t_0)$ for some $t=t_0$. It turns out that $G'(m,0)$ is hard to
analyze so we choose $t_0=\frac{1}{2}$ instead. Then again, as in Theorem \ref{identity-triplet}, we rewrite first
$$F(m,t)=\sum_{j,k,i} (-1)^{j+k+i} {2m+\frac{1}{2}-t \choose 2m+3+j+k}{-2m-1 \choose j}{-2m-1 \choose k}{2m \choose i}
{t+\frac{1}{2} \choose i-k}{t-\frac{1}{2} \choose 2m-j-i}.$$
To compute $F'(m,t)|_{t=1/2}$ we have to analyze two sums:
\be \label{sum-tw1}
Sum(m)=\sum_{i=0}^{2m} (-1)^i {-2m-1 \choose 2m-i}{-2m-1 \choose i}{2m \choose i}.
\ee
This sum was already computed in (\ref{super-rec2}).
We also need
\be  \label{sum-tw2}
TSum(m)=\sum_{i=0}^{2m-1} (-1)^i {-2m-1 \choose 2m-i-1}{-2m-1 \choose i}{2m \choose i}.
\ee
It is not hard to see by using Zeilberger's algorithm  that $TSum(m)=-\frac{1}{2}Sum(m)$.
Putting everything together we get a closed expression for $F'(m,t)|_{t=1/2}$. It is trivial to see that it also equals $G'(m,t)|_{t=1/2}$.
The proof follows.
\end{proof}
\vskip 5mm
By using similar methods we can also prove the following result.
\begin{theorem} \label{tw-nova}
For $m \in \mathbb{N}$, we have
\item[(i)]

\bea && {\rm Res}_{z_1,z_2,z_3} \frac{(1+z_1)^{2m+1/2-t}}{z_1 ^2 } (z_1 z_2 z_3)^{-2m-1} \nonumber \\
&& \cdot (1-z_2/z_1)^{-2m-1}(1-z_3/z_1)^{-2m-1}(z_2-z_3)^{2m}
(1+z_2)^{t+1/2}(1+z_3)^{t-1/2}=0.
\eea

\item[(ii)]
\bea && {\rm Res}_{z_1,z_2,z_3} \frac{(1+z_1)^{2m+1/2-t}}{z_1} (z_1 z_2 z_3)^{-2m-2} \nonumber \\
&& \cdot (1-z_2/z_1)^{-2m-1}(1-z_3/z_1)^{-2m-1}(z_2-z_3)^{2m+2}
(1+z_2)^{t- 1/2}(1+z_3)^{t-1/2} \nonumber \\
& = &- {2m \choose m} ^2 \frac{2m +1}{m+1} { t + m + 1/2 \choose 4 m +2}.   \nonumber \eea
\end{theorem}

\section{Appendix}

Let us briefly outline Zeilberger algorithm method following the book \cite{PWZ}.
One is generally interested in closed expression for finite sum
$$f(n)=\sum_i F(n,i),$$
where $F(n,i)$ is hypergeometric in both arguments (meaning that $F(n,i+1)/F(n,i)$ and $F(n+1,i)/F(n,i)$ are rational functions).
The main idea behind Zeilberger algorithm, also known as the method of {\em creative telescoping}, is to find another function $R(n,i)$ and the following recurrence:
$$\sum_{j=0}^k a_j(n) F(n+j,i)=R(n,i+1)-R(n,i),$$
where $a_j(n)$ are polynomials in $n$.
Assume for a moment that we are able to find such $R(n,i)$, which  is nonzero for finitely many $i$. Then summing
over $i \in \mathbb{Z}$,   yields the recursion
$$\sum_{j=0}^k a_j(n) f(n+j)=0.$$
If there is  such $R(n,i)$ Zeilberger's algorithm can find it, and this part is implemented in various Maple/Mathematica packages
(e.g. {\tt sumtools}). We should say that in all our applications $k=1$ (first order recursions), which can be easily solved and
we get closed expression for $f(n)$.

We illustrate the method on the identity (\ref{super-rec}), or equivalently (\ref{super-rec2}).
Let
$$F(m,i)=(-1)^i {-2m-1 \choose i}{-2m-1 \choose 2m-i}{2m \choose i}.$$
Zeilberger's algorithm gives
$$R(m,i)=G(m,i)F(m,i),$$
$$ G(m,i)=\frac{i^2 (-i+4m+1) \biggl((m+1)i^4-2(m+1)(5m+4)i^3+(m+1)(120m^2+168m+59)i^2}
{4(m+1)(2m+1)(-i+2m+1)^2(-i+2m+2)^2}$$
$$ \frac{-2(m+1)(260m^3+558m^2+395m+92)i+4(m+1)(172m^4+506m^3+548m^2+259m+45)\biggr)  }{4(m+1)(2m+1)(-i+2m+1)^2(-i+2m+2)^2}.$$
It is easy to prove the identity
\bea && -(m+1)^2 F(m+1,i)-3(3m+1)(3m+2)F(m,i)=R(m,i+1)-R(m,i). \label{zeil-1} \eea
We shall also need
\bea
&&  (m+1) ^2 ( F( m+1,2m) + F(m+1,2m+1) + F(m+1,2m+2) )  \nonumber  \\ && \quad + 3 (3m+1) (3m+2) F(m,2m) = R(m,2m). \label{zeil-2} \eea
Now we sum the equation (\ref{zeil-1})  over $i \in \{0,..,2m-1  \}$ and use (\ref{zeil-2}).
We obtain
$$-(m+1)^2 f(m+1) -3(3m+1)(3m+2)f(m)=0$$
as desired. This immediately gives
$$f(m)= \sum_{i=0}^{2m} (-1)^i {-2m-1 \choose i}{-2m-1 \choose 2m-i}{2m \choose i}=\frac{(-1)^m (3m)!}{m!^3}.$$


\begin{thebibliography}{FGST2}



\bibitem[1]{A-1997} D. Adamovi\'c, Rationality of Neveu-Schwarz vertex operator superalgebras, Internat. Math. Res. Notices No. 17 (1997), 865--874.

\bibitem[2]{A-2003} D. Adamovi\'{c},
Classification of irreducible modules of certain subalgebras of
free boson vertex algebra, J. Algebra 270 (2003) 115--132.

\bibitem[3]{AdM-2007} D. Adamovi\'c and A. Milas, Logarithmic intertwining operators
and $\mathcal{W}(2,2p-1)$-algebras, {\em Journal of Math. Physics}
{\bf 48}, 073503 (2007).

\bibitem[4]{AdM-triplet} D. Adamovi\'c and A. Milas, On the triplet vertex algebra
$\mathcal{W}(p)$, {\em Advances in Math.} 217 (2008) 2664-2699,
{\tt arxiv:0707.1857}.


\bibitem[5]{AdM-tstriplet}   D. Adamovi\'c and A. Milas, The $N=1$
 triplet vertex operator superalgebras: twisted sector, SIGMA 4 (2008) 24 pages, {\tt arXiv:0806.3560}.

\bibitem[6]{AdM-striplet} D. Adamovi\'c and A. Milas, The $N=1$
triplet  vertex operator superalgebras,  {\em Comm. Math. Phys.} 288 (2009) 225-270, {\tt  arXiv:0712.0379}.



\bibitem[7]{AdM-2009}  D. Adamovi\'c and A. Milas, On $\W$-algebras associated to $(2,p)$ minimal models and their representations,  Int. Math. Res. Not. IMRN 2010, no. 20, 3896�-3934,  {\tt arXiv:0908.4053}.

\bibitem[8]{AdM-2010} D. Adamovi\'c and A. Milas, On $W$-algebra extensions of $(2,p)$ minimal models: $p >3$ ,  {\tt  arXiv:1101.0803}.


\bibitem[9]{FFL} B. Feigin, E. Feigin, and P. Littelmann, Zhu's algebras, $C_2$-algebras and abelian radicals,  J. Algebra 329 (2011) 130-146,  {\tt arXiv:0907.3962}.

\bibitem[10]{FL} E. Feigin, and P. Littelmann, Zhu's algebra and the $C_2$-algebra in the symplectic and the orthogonal cases,  J. Phys. A: Math. Theor. 43 (2010) 135206, {\tt arXiv:0911.2957v1}.



\bibitem[11]{FGST-triplet}  B.L. Feigin, A.M. Ga\u\i nutdinov, A. M. Semikhatov, and I. Yu Tipunin, Kazhdan--Lusztig correspondence for the representation category of the triplet W-algebra in logarithmic CFT
, Theor.Math.Phys. 148 (2006) 1210-1235; Teor.Mat.Fiz. 148 (2006) 398--427.

\bibitem [12]{FGST-log} B.L. Feigin, A.M. Ga\u\i nutdinov, A. M. Semikhatov, and I. Yu Tipunin,
 Logarithmic extensions of minimal
models: characters and modular transformations, Nucl. Phys. B 757
(2006) 303-343.

\bibitem[13]{FGST-log-kl}  B.L. Feigin, A.M. Ga\u\i nutdinov, A. M. Semikhatov, and I. Yu Tipunin, Kazhdan--Lusztig-dual quantum group for logarithmic extensions of Virasoro minimal models,  {\em Jour. of Math. Physics } {\bf 48}  032303 (2007).

\bibitem[14]{FHST} J. Fuchs, S. Hwang, A.M. Semikhatov and I. Yu. Tipunin,
Nonsemisimple Fusion Algebras and the Verlinde Formula, { Comm.
Math. Phys.} {\bf 247} (2004), no. 3, 713--742.


\bibitem[15]{GG} M. Gaberdiel and T. Gannon, Zhu's algebra, the $C_2$-algebra, and twisted modules, in: Contemp. Math., vol. 497, Amer.
Math. Soc., 2009, pp. 65�-78.

\bibitem[16]{GRW} M. Gaberdiel, I. Runkel and S. Wood, Fusion rules and boundary conditions in the $c=0$ triplet model,   J. Phys. A 42 (2009), no. 32, 325403, 43 pp.
{\tt arxiv :0905.0916}.


\bibitem[17]{HLZ} Y.-Z. Huang, J. Lepowsky, and L. Zhang., Logarithmic tensor product theory for generalized modules for a conformal
vertex algebra, {\tt arXiv:0710.2687}.


\bibitem[18]{KW} V. G. Kac and W. Wang, Vertex operator superalgebras and their representations, in: Contemp. Math., vol. 175, Amer.
Math. Soc., 1994, pp. 161�-191.

\bibitem[19]{M} A. Milas,   Characters, supercharacters and Weber modular functions, {\em J. Reine Angew. Math.}  {\bf 608} (2007),
35--64.


\bibitem[20]{Miy} M. Miyamoto,
Modular invariance of vertex operator algebras satisfying $C\sb
2$-cofiniteness. {\em Duke Math. J.} {\bf 122} (2004), 51--91.

\bibitem[21]{PWZ} M. Petkov\v sek, H. Wilf and D. Zeilberger, {\em A=B},  A.K. Peters, Wellesley, MA (1996).\\ available at  {\tt http://www.math.rutgers.edu/ ~ zeilberg/AeqB.pdf}

\bibitem[22]{Xu}  Xu X., Introduction to vertex operator superalgebras and their modules, Mathematics and Its Applications,
Vol. 456, Kluwer Academic Publishers, Dordrecht, 1998.

\bibitem[23]{Z} Zhu, Y.-C., Modular invariance of characters of vertex operator algebras, J. Amer. Math. Soc. 9 (1996), 237--302.


\end{thebibliography}
\end{document}